\documentclass[a4paper,12pt]{amsart}

\usepackage[centertags]{amsmath}
\usepackage{amsfonts,amssymb,amstext,amsthm,newlfont,latexsym,stmaryrd}
\usepackage[left=3.0cm,right=3.0cm,top=3.7cm,bottom=3.7cm]{geometry}
\usepackage{tikz,hyperref}

\theoremstyle{plain}
\newtheorem{theorem}{Theorem}[section]
\newtheorem{corollary}[theorem]{Corollary}
\newtheorem{lemma}[theorem]{Lemma}
\newtheorem{proposition}[theorem]{Proposition}
\theoremstyle{definition}
\newtheorem{definition}[theorem]{Definition}
\newtheorem{example}[theorem]{Example}
\theoremstyle{remark}

\newtheorem*{ack*}{Acknowledgment}

\begin{document}

\title[Relative transition classes]{The multiplicative inequality for class degrees via relative transition classes}

\author[S. Hong]{Soonjo Hong}
\address{Centro de Modelamiento Matem\'atico \\
    Universidad de Chile \\
    Av. Blanco Encalada 2120, Piso 7 \\
    Santiago de Chile \\
    Chile}
\email{hsoonjo@dim.uchile.cl}

\date{}
\begin{abstract}
  Generalizing the notion of the degree of a finite-to-one factor code from a shift of finite type, the class degree of a possibly infinite-to-one factor code extends many important properties of degree. In this paper, introducing relative class degree, we study how class degrees change as two factor codes are composited, in comparison with degrees.
\end{abstract}
\maketitle

\section{Introduction and backgrounds}

  A finite-to-one factor code $\pi$ from an irreducible shift of finite type onto a sofic shift $Y$ exhibit several nice behaviours. It preserves entropy, forbids graph diamonds, and, in particular, has \emph{degree}. The degree $d$ of $\pi$ is the minimal cardinality of the set of preimages of a point in $Y$. The well-known properties of $\pi$ related to the degree include: Every bi-transitive point in $Y$ has exactly $d$ preimages. They are bi-transitive in $X$ themselves and two distinct ones among the preimages are mutually separated and so on \cite{CovenP74,Hed69,KitMT91}.

  Much less is known about general, not necessarily finite-to-one factor codes. One notion called \emph{class degree} was devised to extend degree to general factor codes \cite{AllQ13}. It is defined to be the minimal cardinality of certain equivalence classes partitioning the preimages of a point in $Y$. The notion was shown to generalize degree in the sense that a finite-to-one factor code has the same class degree as its degree, and that for a 1-block factor code $\pi$ with the class degree $d$ the followings hold: Over every right transitive point in $Y$ exist exactly $d$ transition classes. Each of them contains a right transitive point of $X$ and two points from two distinct ones among the $d$ classes are mutually separated \cite{AllQ13,AllHJ13}.
  
  The multiplicative equality for the degree of a composition of finite-to-one factor codes, however, does not hold for class degree (See \hyperref[ex:multiplication_fails]{Example \ref{ex:multiplication_fails}}). Instead a multiplicative inequality was proven in \cite{AllHJ13} : Let $X$ and $Y$ be irreducible shifts of finite type, $Z$ a sofic shift and $\phi:X\to Y$, $\psi:Y\to Z$, $\pi:X\to Z$ factor codes with $\pi=\psi\circ\phi$. Then the class degree of $\pi$ is no greater than the multiplication of the class degrees of $\psi$ and $\phi$.
  
  In this paper we study the class degrees of the compositions of factor codes further. We look into a structure called \emph{relative transition class} which lies under the compositions. We reprove the multiplicative inequality, find how relative transition classes are organized in relation with transition classes and establish some important cases in which the multiplicative equality holds. The main result is as follows:
  \begin{theorem}\label{thm:main}
    Let $X$ and $Y$ be irreducible shifts of finite type, $Z$ a sofic shift and $\phi:X\to Y$, $\psi:Y\to Z$, $\pi:X\to Z$ factor codes with $\pi=\psi\circ\phi$. Then the class degree of $\pi$ is the class 	degree of $\psi$ multiplied by the class degree of $\phi$ relative to $\psi$.
  \end{theorem}

  We introduce basic terminology and known results. For more details on symbolic dynamics, see \cite{LM}.

  An \emph{alphabet} $\mathcal{A}$ is a finite set whose elements are called \emph{symbols}. A \emph{block} of length $n\in\mathbb{Z}^+$ over $\mathcal{A}$ is a concatenation of $n$ symbols from $\mathcal{A}$. The \emph{empty} block of length 0 is assumed to exist. A \emph{shift space} $X$ over $\mathcal{A}$ is a closed $\sigma$-invariant subset of $\mathcal{A}^\mathbb{Z}$, endowed with the product topology, where $\sigma:\mathcal{A}^\mathbb{Z}\to\mathcal{A}^\mathbb{Z}$ is given by $\sigma(x)|_i=x|_{i+1}$. If for any $u,v$ in the set $\mathcal{B}(X)$ of all blocks occurring in the points of $X$ there is a block $w$ with $uwv$ in $\mathcal{B}(X)$, then $X$ is said to be \emph{irreducible}. In this paper every block starts with coordinate 1 so that a block $w$ of length $l$ is written to be $w|_{[1,l]}$.

  A point $x$ of $X$ is \emph{right transitive} if every block in $\mathcal{B}(X)$ appears in $x|_{[0,\infty)}$.  Two points $x$ and $x'$ of $X$ are \emph{right asymptotic} if $x|_{[n,\infty)}=x'|_{[n,\infty)}$ for some $n\in\mathbb Z$. \emph{Left transitive points} and \emph{left asymptotic points} are defined analogously. A point is \emph{bi-transitive} if it is both left and right transitive and two points are \emph{bi-asymptotic} if they are both left and right asymptotic. If $x|_i\ne x'|_i$ for all $i\in\mathbb{Z}$, then $x$ and $x'$ are said to be \emph{mutually separated} .

  We call a continuous $\sigma$-commuting map between shift spaces a \emph{code}, a surjective one a \emph{factor code} and a bijective one a \emph{conjugacy}. Shift spaces $X$ and $Y$ are \emph{conjugate} if there is a conjugacy between them. A code $\pi:X\to Y$ is \emph{1-block} if $x|_0$ determines $\pi(x)|_0$ for any point $x$ in $X$. A 1-block code $\pi:X \to Y$ induces a map from $\mathcal{B}(X)$ into $\mathcal{B}(Y)$, also denoted by $\pi$. Every code $\pi:X\to Y$ is \emph{recoded} to some 1-block code $\tilde\pi:\tilde X\to Y$, where $\tilde X$ is conjugate to $X$ by a conjugacy commuting with $\pi$. We call $\pi$ \emph{finite-to-one} if $\pi^{-1}(y)$ is finite for all $y$ in $Y$. An image of an irreducible shift space under a factor code is also irreducible.
  
  If for some $k\ge0$ every block of length $k$ in $\mathcal{B}(X)$ satisfies that whenever $uw$ and $wv$ are in $\mathcal{B}(X)$ so is $uwv$, then $X$ is called a (\emph{$k$-step}) \emph{shift of finite type}. Up to conjugacy, every shift of finite type is 1-step. A \emph{sofic shift} is an image of a shift of finite type by a factor code.

\section{Relative transition class and relative class degreess}\label{sec:relative_class_degrees}
  
  Let $\pi:X\to Y$ be a 1-block factor code from an irreducible 1-step shift of finite type $X$ over $\mathcal{A}$ onto a sofic shift $Y$, unless stated otherwise.

  \begin{definition}\label{defn:degree}
    Given $m\in\mathbb Z,y$ in $Y$ and $x,x'\in X$ with $\pi(x)=\pi(x')$ a $(\pi,m)$-\emph{bridge from $x$ to $x'$} is a point $\vec x$ in $X$ such that $x|_{(-\infty,m]}=\vec x|_{(-\infty,m]},\vec x|_{[n,\infty)}=x'|_{[n,\infty)}$ and $\pi(\vec{x})=\pi(x)=\pi(x')$ for some $n>m$. A pair of $(\pi,m)$-bridges from $x$ to $x'$ and from $x'$ to $x$ is called a {\em two-way} $(\pi,m)$-{\em bridge between $x$ and} $x'$. A $\pi$-\emph{transition from $x$ to $x'$} is a sequence $\{\vec{x}_m\}_{m\in\mathbb Z}$ of $(\pi,m)$-bridges from $x$ to $x'$. When there is a $\pi$-transition from $x$ to $x'$, we write that $x\to_\pi x'$.
    
    We say that $x$ and $x'$ are $\pi$-\emph{equivalent} and write that $x\sim_\pi x'$ if $x\to_\pi x'$ and $x'\to_\pi x$. The equivalence class $[x]_\pi$ of $x$ up to $\sim_\pi$ is called a $\pi$-\emph{(transition) class}. Set $\llbracket y\rrbracket_\pi=\{[x]_\pi\mid x\in\pi^{-1}(y)\}$ and $d_\pi(y)=|\llbracket y\rrbracket_\pi|$ for $y$ in $Y$. The \emph{class degree} $d_\pi$ of $\pi$ is defined to be $d_\pi=\min_{y\in Y}d_\pi(y).$
  \end{definition}

  If $\pi$ is finite-to-one, then $d_\pi$ equals the {\em degree} of $\pi$, which is the minimal cardinality of the set of preimages of a point in $Y$. In most cases, the term \emph{class} will be omitted, which are justified as class degree generalizes degree in a natural way \cite{AllQ13,AllHJ13}.

  Degree can be dealt locally in terms of \emph{routability}.
  \begin{definition}\label{defn:transition_block}
    A block $w$ in $\mathcal{B}(Y)$ is said to be $\pi$-\emph{presented through $M$ at $n$} for some $1\le n\le|w|$ and $M\subset\mathcal{A}$ if for every $u$ in $\pi^{-1}(w)$ there is some $v$ in $\pi^{-1}(w)$ with $u|_1=v|_1,u|_{|w|}=v|_{|w|}$ and $v|_n$ in $M$, in which case $u$ is said to be $\pi$-\emph{routable through $v|_n$ at $n$}. Define the $\pi$-\emph{depth} $d_\pi(w)$ of $w$ by
    \[ d_\pi(w)=\min\{|M|\mid w\text{ is $\pi$-presented through $M\subset\mathcal{A}$ at $n$ for some }1\le n\le|w|\}.\]
    A $\pi$-presented block $w$ is said to be \emph{minimal} if  $d_\phi(w)=\min_{v\in\mathcal{B}(Y)}d_\phi(v)$.
  \end{definition}
  
  It was shown in \cite{AllQ13} that $d_\phi=d_\phi(w)$ when $w$ is minimal. The following are among the most important properties of transition classes and class degrees.
  \begin{theorem}\label{thm:class_transitive_point}\cite{AllQ13,AllHJ13}
    Let $X$ be a 1-step irreducible shift of finite type and $Y$ a sofic shift with a 1-block factor code $\pi:X\to Y$. Let $y$ in $Y$ be right transitive (resp., bi-transitive). Then the following hold:
    \begin{enumerate}
      \item There are exactly $d_\pi$ $\pi$-classes over $y$.
      \item Any $\pi$-class over $y$ contains a right transitive (resp., bi-transitive) point.
      \item Two points from distinct $\pi$-classes over $y$  are mutually separated.
    \end{enumerate}
  \end{theorem}

  For more details on class degrees, see \cite{AllHJ13,AllQ13}. From now on our interest is moving on relative transition and relative class degree. Throughout the rest of the paper $X$ and $Y$ are assumed to be 1-step irreducible shifts of finite type and $Z$ a sofic shift with 1-block factor codes $\phi:X\to Y,\psi:Y\to Z$ and $\pi=\psi\circ\phi$, unless stated otherwise.
  \begin{definition}\label{defn:relative_class_degree}
    Given $m\in\mathbb Z$ and $x,x'\in X$ with $\phi(x)\sim_\psi\phi(x')$ a $(\phi/\psi,m)$-\emph{bridge from $x$ to $x'$} is a point $\vec x\in X$ such that $x|_{(-\infty,m]}=\vec x|_{(-\infty,m]},\vec x|_{[n,\infty)}=x'|_{[n,\infty)}$ and $\phi(x)\sim_\psi\phi(x')\sim_\psi\phi(\vec x)$ for some $n>m$.
    A $\phi/\psi$-\emph{transition from $x$ to $x'$} is a sequence $\{\vec{x}_m\}_{m\in\mathbb Z}$ of $(\phi/\psi,m)$-bridges from $x$ to $x'$. When there is a $\phi/\psi$-transition from $x$ to $x'$, we write that $x\to_{\phi/\psi} x'$.
    
    We say that $x$ and $x'$ are $\phi/\psi$-\emph{equivalent} and write that $x\sim_{\phi/\psi} x'$ if $x\to_{\phi/\psi} x'$ and $x'\to_{\phi/\psi} x$. The equivalence class $[x]_{\phi/\psi}$ of $x$ up to $\sim_{\phi/\psi}$ is called a $\phi/\psi$-\emph{(transition) class}. Set $\llbracket y\rrbracket_{\phi/\psi}=\{[x]_{\phi/\psi}\mid x\in\phi^{-1}(y)\}$ and $d_{\phi/\psi}(y)=|\llbracket y\rrbracket_{\phi/\psi}|$ for $y$ in $Y$. The \emph{degree} $d_{\phi/\psi}$ of $\phi$ \emph{relative to} $\psi$ is defined to be $d_{\phi/\psi}=\min_{y\in Y}d_{\phi/\psi}(y).$
  \end{definition}

  If $\psi$ is the identity then $[x]_{\phi/\psi}$ is just $[x]_\phi$. Also note that relative degrees are always smaller than or equal to degrees. We define \emph{relative routability} as well.
  \begin{definition}\label{defn:relative_transition_block}
    A block $w$ in $\mathcal{B}(Y)$ is said to be $\phi/\psi$-\emph{presented through $M$ at $n$} for some $1\le n\le|w|$ and $M\subset\mathcal{A}$ if for every $u\in\phi^{-1}(w)$ there is some $v$ in $\pi^{-1}(\psi(w))$ with $u|_1=v|_1,u|_{|w|}=v|_{|w|}$ and $v|_n$ in $M$, in which case $u$ is said to be $\phi/\psi$-\emph{routable through $v|_n$ at $n$}. Define the $\phi/\psi$-\emph{depth} $d_{\phi/\psi}(w)$ of $w$ by
    \[ d_\pi(w)=\min\{|M|\mid w\text{ is $\phi/\psi$-presented through $M\subset\mathcal{A}$ at $n$ for some }1\le n\le|w|\}.\]
    A $\phi/\psi$-presented block $w$ is said to be \emph{minimal} if $d_{\phi/\psi}(w)=\min_{v\in\mathcal{B}(Y)}d_{\phi/\psi}(v)$.
  \end{definition}

  If both $\phi$ and $\psi$ are finite-to-one then we have the equality $d_\pi=d_\phi d_\psi$. In infinite-to-one cases the equality fails in as a simplest example as follows:

  \begin{example}\label{ex:multiplication_fails}\cite{AllHJ13}
    Let $X$ be the full shift over $\mathcal{A}=\{0,1\}$ and $Y = \{ 0^\infty \}$ and consider the trivial map $\pi:X\to Y$. Letting $\psi = \pi$ and $\phi:X\to X$ by $\phi(x)|_i = x|_i + x|_{i+1} \mod 2$, we have $\pi=\psi\circ\phi$. However, $d_\pi=1<d_\phi d_\psi=2$.
  \end{example}

  Still one will see that an inequality holds for relative degrees. The first relation we will see between transition classes and relative transition classes is that they are indeed the same regarded as the subsets of $X$.
  \begin{theorem}\label{thm:classes_and_relative_classes}
    Let $X$ and $Y$ be 1-step irreducible shifts of finite type and $Z$ a sofic shift with 1-block factor codes $\phi:X\to Y,\psi:Y\to Z$ and $\pi=\psi\circ\phi$. For all $x$ in $X$ we have that $[x]_\phi\subset[x]_{\phi/\psi}=[x]_\pi$. Any two points in a single $\pi$-class are sent by $\phi$ into the same $\psi$-classes, and $[x]_{\phi/\psi}=\bigsqcup_{x'\sim_{\phi/\psi} x}[x']_\phi$.
  \end{theorem}
  \begin{proof}
    Fix any $z$ in $Z,x$ in $\pi^{-1}(z)$ and choose $x'$ from $[x]_{\phi/\psi}$. Given $m\in\mathbb Z$ there is a $(\phi/\psi,m)$-bridge $\vec x_m$ from $x$ to $x'$. Since $\phi(x)\sim_\psi\phi(x')\sim_\psi\phi(\vec x_m)$, applying $\psi$ to the all sides $\pi(x)=\pi(x')=\pi(\vec x_m)=z$. So $\{\vec x_m\}_{m\in\mathbb Z}$ is a $\pi$-transition from $x$ to $x'$. Similarly, $x'\to_\pi x$, and $[x]_{\phi/\psi}\subset[x]_\pi$.
  
    To show the other inclusion, choose $x'$ from $[x]_\pi$. Given $m\in\mathbb Z$ find a $(\pi,m)$-bridge $\vec x_m$ from $x$ to $x'$. Clearly, $\phi(x)\sim_\psi\phi(x')\sim_\psi\phi(\vec x_m)$ so that $\{\vec x_m\}$ is a $\phi/\psi$-transition, too. Similarly, $x'\to_{\phi/\psi} x$, and $[x]_{\phi/\psi}=[x]_\pi$.
  
    Finally, fix any $y\in Y$ and choose $x,x'$ from $\phi^{-1}(y)$ with $x\sim_\phi x'$. It is easy to see that $x\sim_\pi x'$ and $[x]_\phi\subset[x]_\pi$. The second statement follows immediately.
  \end{proof}

  Next we establish the relation between relative degree and relative depth, and further show that transitive points realize relative degree, as in the case of degree.

  \begin{lemma}\label{lmm:transition_blk_makes_a_transition}
    Let $X$ and $Y$ be 1-step irreducible shifts of finite type and $Z$ a sofic shift with 1-block factor codes $\phi:X\to Y,\psi:Y\to Z$ and $\pi=\psi\circ\phi$. Let $y$ in $Y$ be such that $w=y|_{(0,|w|]}$ is $\phi/\psi$-presented through $M$ at $n$, and $x$ and $x'$ in $\phi^{-1}([y]_\psi)$ satisfy that both $x|_{(0,|w|]}$ and $x'|_{(0,|w|]}$ are $\phi/\psi$-routable through $a$ in $M$ at $n$. Then there is a two-way $(\phi/\psi,0)$-bridges between $x'$ and $x$.
  \end{lemma}
  \begin{proof}
    Say $u=x|_{(0,|w|]}$ and $u'=x'|_{(0,|w|]}$, for the convenience. There are two blocks $v$ and $v'$ in $\pi^{-1}(\psi(w))$ such that $u|_1=v|_1,u|_{|w|}=v|_{|w|}$ and $u'|_1=v'|_1,u'|_{|w|}=v'|_{|w|}$ and $v|_n=v'|_n=a$. Define $\vec{x}$ by
    \[ \vec{x}|_{(-\infty,0]}=x|_{(-\infty,0]},\vec{x}|_{(0,|w|]}=v|_1\cdots v|_{n-1}av'|_{n+1}\cdots v'|_{|w|}\text{ and }\vec{x}|_{(|w|,\infty)}=x'|_{(|w|,\infty)}. \]
    As $\vec{x}|_1=v|_1=u|_1=x|_1$ and $\vec{x}|_{|w|}=v'|_{|w|}=u'|_{|w|}=x'|_{|w|}$, the point $\vec{x}$ is in $X$. Since $\vec x$ and $x'$ are right asymptotic, so are $\phi(\vec{x})$ and $\phi(x')$. They also have the same image under $\psi$ as $\pi(v)=\pi(v')=\psi(w)=\pi(u)=\pi(u')$. Therefore $\phi(\vec x)\sim_\psi\phi(x')\sim_\psi\phi(x)$ and $\vec{x}$ turns out to be a $(\phi/\psi,0)$-bridge from $x$ to $x'$. A $(\phi/\psi,0)$-bridge from $x'$ to $x$ is found in a similar way.
  \end{proof}

  Denote by $C|_A$ the set $\{x|_A\mid x\in C\}$ for a subset $C$ of $X$ and a subset $A$ of $\mathbb{Z}$. The following theorem will be useful for \hyperref[lmm:degree_and_depth]{Lemma \ref{lmm:degree_and_depth}}. It is shown in the proof of Theorem 4.22 of \cite{AllQ13}. \hyperref[lmm:degree_and_depth]{Lemma \ref{lmm:degree_and_depth}} and \hyperref[lmm:degree_and_depth]{Proposition \ref{prop:reldegree_of_a_point}} are reminiscent of the theorem.

  \begin{theorem}\label{thm:degree_and_depth}\cite{AllQ13}
    Let $\pi:X\to Z$ be a 1-block factor code from an irreducible shift of finite type $X$. Let $z$ be in $Z$ and $N\in\mathbb{N}$. Then there are integers $m'>m>N$ and $0<n\le m'-m$ such that for any $C$ in $\llbracket z\rrbracket_\pi$ we have a symbol $a$ in the alphabet $\mathcal{A}_X$ of $X$ through which all the blocks in $C|_{[m,m']}$ are $\pi$-routable at $n$.
  \end{theorem}

  \begin{lemma}\label{lmm:degree_and_depth}
    Let $X$ and $Y$ be 1-step irreducible shifts of finite type and $Z$ a sofic shift with 1-block factor codes $\phi:X\to Y,\psi:Y\to Z$ and $\pi=\psi\circ\phi$. Then for any $y$ in $Y$ there is a block $v$ occurring in $y$ such that $d_{\phi/\psi}(v)\le d_{\phi/\psi}(y)$.
  \end{lemma}
  \begin{proof}
    Let $z=\psi(y)$. By \hyperref[thm:degree_and_depth]{Theorem \ref{thm:degree_and_depth}} for all $k\in\mathbb{N}$ there are integers $l_k<r_k,1\le n_k\le r_k-l_k$ and a set $M_k\subset\mathcal{A}_X$ with $|M_k|\le d_\pi(z)$ such that for $I_k=[l_k,r_k]$ and any $C$ in $\llbracket z\rrbracket_\pi$ we have a symbol in $M_k$ through which all the blocks in $C|_{I_k}$ are $\pi$-routable at $n$.

    List the elements of $\llbracket y\rrbracket_{\phi/\psi}$ as $C_1,\cdots,C_d$ where $d=d_{\phi/\psi}(y)$. By \hyperref[thm:classes_and_relative_classes]{Theorem \ref{thm:classes_and_relative_classes}} they are $\pi$-classes over $z$ at the same time. We may assume that $I=I_1=[1,l]$. Put $n=n_1,M=M_1$ and $v=y|_I$. All the blocks $u$ in $C_j|_I$ are $\pi$-routable through a single symbol of $M$ at $n$ for each $1\le j\le d$. Yet it is possible that $\phi^{-1}(v)\setminus\bigcup_jC_j|_I$ is not empty. To avoid the problem extend $I$ to $I'=[1-m,l+m]$ for sufficiently large $m$. By compactness there are some $m>0,1\le p\le m+1$ and $1\le q\le m+1$ such that $\pi^{-1}(y|_{I'})|_{[p,l+2m-q+1]}=\bigcup_jC_j|_{[-m+p,l+m-q+1]}$. Reset to be $y|_{I'}$ the block $v$ still is $\pi$-presented through $M'=M\cap\bigcup_jC_{j=1}^d|_n$ at $N+n$. As each $M\cap C_j$ has a unique element $d_{\phi/\psi(v)}\le|M'|\le d$ and $v$ is a desired block.
  \end{proof}

  A \emph{recurrent} point of a shift space is a point every block of which occurs infinitely often to its right. A right transitive point is always recurrent.

  \begin{proposition}\label{prop:reldegree_of_a_point}
    Let $X$ and $Y$ be 1-step irreducible shifts of finite type and $Z$ a sofic shift with 1-block factor codes $\phi:X\to Y,\psi:Y\to Z$ and $\pi=\psi\circ\phi$. Let $y$ be a recurrent point of $Y$ and $d'=\min\{d_{\phi/\psi}(w)\mid w\text{ is a block occurring in }y\}$. Then $d_{\phi/\psi}(y)=d'$.
  \end{proposition}
  \begin{proof}
    Put $d=d_{\phi/\psi}(y)$. \hyperref[lmm:degree_and_depth]{Lemma \ref{lmm:degree_and_depth}} induces the inequality $d'\le d$. Let $w$ be a block that occurs in $y$ and is $\phi/\psi$-presented through $M$ at $n$ with $M=\{a_1,\cdots,a_{d'}\}$. Trivially $d_{\phi/\psi}(w)=d'$. Let $y|_{(n_i,n_i+|w|]}=w$ for increasing numbers $n_i\in\mathbb{N}$ with $n_{i+1}>n_i+|w|$. Choose any $d'+1$ points $x_0,\cdots,x_{d'}$ in $\phi^{-1}([y]_\psi)$. We claim that there are $0\le j<k\le d'$ with $x_j\sim_{\phi/\psi} x_k$. The claim being true shows that $d\le d'$.
  
    For every $i\in\mathbb{N}$ each of the blocks $x_0|_{I_i},\cdots,x_{d'}|_{I_i}$ is $\pi$-routable through one of some symbols $m_i(0),\cdots,m_i(d')$ at $n$, respectively, where $I_i=(n_i,n_i+|w|]$ and $m_i$ maps $\{0,\cdots,d'\}$ into $M$. For every $i\in\mathbb{N}$ there are $0\le j_i<k_i\le d'$ with $m_i(j_i)=m_i(k_i)$ since $|M|=d'$. By the pigeonhole principle there must be $0\le j<k\le d'$ such that $(j,k)=(j_i,k_i)$ for infinitely many $i$'s. By \hyperref[lmm:transition_blk_makes_a_transition]{Lemma \ref{lmm:transition_blk_makes_a_transition}} for each $i$ with $(j_i,k_i)=(j,k)$ there exists a two-way $(\phi/\psi,n_i)$-bridge between $x_j$ and $x_k$. Hence $x_j\sim_{\phi/\psi}x_k$.
  \end{proof}

  \begin{corollary}\label{cor:degree_right_transitive_point}
    Let $X$ and $Y$ be 1-step irreducible shifts of finite type and $Z$ a sofic shift with 1-block factor codes $\phi:X\to Y,\psi:Y\to Z$ and $\pi=\psi\circ\phi$. Let $y$ be a right transitive point of $Y$. Then $d_{\phi/\psi}=d_{\phi/\psi}(y)$.
  \end{corollary}
  \begin{proof}
    Any $\phi/\psi$-presented block $w$ with the smallest $\phi/\psi$-depth among all blocks in $\mathcal{B}(Y)$ occurs in $y$ as it is right transitive. So $d_{\phi/\psi}(y)=d_{\phi/\psi}(w)$. Also any point $y'$ in $Y$ has by \hyperref[lmm:degree_and_depth]{Lemma \ref{lmm:degree_and_depth}} a $\phi/\psi$-presented block $v$ the $\phi/\psi$-depth of which is no greater than $d_{\phi/\psi}(y')$. Then $d_{\phi/\psi}(y')\ge d_{\phi/\psi}(v)\ge d_{\phi/\psi}(w)=d_{\phi/\psi}(y)$. Hence $d_{\phi/\psi}=d_{\phi/\psi}(y)$.
  \end{proof}

  Combining \hyperref[thm:class_transitive_point]{Theorem \ref{thm:class_transitive_point}}, \hyperref[thm:classes_and_relative_classes]{Theorem \ref{thm:classes_and_relative_classes}} and \hyperref[cor:degree_right_transitive_point]{Corollary \ref{cor:degree_right_transitive_point}} \hyperref[prop:relative_class_right_transitive_point]{Proposition \ref{prop:relative_class_right_transitive_point}} is obtained.

  \begin{proposition}\label{prop:relative_class_right_transitive_point}
    Let $X$ and $Y$ be 1-step irreducible shifts of finite type and $Z$ a sofic shift with 1-block factor codes $\phi:X\to Y,\psi:Y\to Z$ and $\pi=\psi\circ\phi$. Let $y$ in $Y$ be right transitive (resp., bi-transitive). Then the following hold:
    \begin{enumerate}
      \item There are exactly $d_{\phi/\psi}$ $\phi/\psi$-classes over $y$.
      \item Any $\phi/\psi$-class over $y$ contains a right transitive (resp., bi-transitive) point.
      \item Two points from distinct $\phi/\psi$-classes over $y$  are mutually separated.
    \end{enumerate}
  \end{proposition}

  \hyperref[thm:division_multiplicity]{Theorem \ref{thm:division_multiplicity}} is the most important result of this paper. It describes how relative transition classes are organized over transition classes. \hyperref[thm:block_partition]{Theorem \ref{thm:block_partition}} is a key ingredient to its proof. Given a block $w$ of $Z$ let $d^*(w,i)=|\{u|_i\mid\pi(u)=w\}|$. Then $w$ is called a {\em magic block} with a {\em magic coordinate} $i$ if $d^*(w,i)$ is the smallest among all blocks of $Z$ and numbers $i$.

  \begin{theorem}\label{thm:block_partition}\cite{AllHJ13}
      Let $X$ be a 1-step irreducible shift of finite type and $Z$ a sofic shift with a 1-block factor code $\pi:X\to Y$. Let $w$ be a magic block of $\pi$. Then there is a partition $\mathcal{P}$ of $\pi^{-1}(w)$ into $d_\pi$ subsets such that for any bi-transitive point $z$ in $Z$ and $i\in\mathbb{Z}$ with $z|_{(i,i+|w|]} = w$ we have a bijection $\rho_{z,i}:\llbracket z\rrbracket_\pi\to\mathcal{P}$ with $D|_{(i,i+|w|]} \subset\rho_{z,i}(D),D\in\llbracket z\rrbracket_\pi$.
  \end{theorem}

  \begin{theorem}\label{thm:division_multiplicity}
    Let $X$ and $Y$ be irreducible shifts of finite type, $Z$ a sofic shift and $\phi:X\to Y$, $\psi:Y\to Z$, $\pi:X\to Z$ 1-block factor codes with $\pi=\psi\circ\phi$. Then $d_{\phi/\psi}$ divides $d_\phi$ and $d_\pi=d_\psi d_{\phi/\psi}\le d_\psi d_\phi$.
  \end{theorem}
  \begin{proof}  
    Degrees and relative degrees are attained over bi-transitive points. Also all transition classes and relative transition classes over bi-transitive points contain bi-transitive points. Hence we may concern only bi-transitive points from now on.
  
    Fix a bi-transitive point $y=y_1$ in $Y$. Then $z=\psi(y)$, too, is bi-transitive. Over $z$ $\psi$-classes $D_1=[y]_\psi,\cdots,D_{d_\psi}$ exist in $Y$. Pick up bi-transitive points $y_2,\cdots,y_{d_\psi}$ from each $D_i,2\le i\le d_\psi$. Over each $y_i$ $\phi$-classes $C_{i,1},\cdots,C_{i,d_\phi}$ exist in $X$ for $1\le i\le d_\psi$. Pick up bi-transitive points $x_{i,j}$ from each $C_{i,j}$ for $1\le i\le d_\psi$ and $1\le j\le d_\phi$.
  
    Note that $\bigcup_{i,j}C_{i,j}=\pi^{-1}(z)$. According to \hyperref[thm:classes_and_relative_classes]{Theorem \ref{thm:classes_and_relative_classes}} each of the $\pi$-classes $E_1,\cdots,E_d$ over $z$ is a disjoint union of some $C_{i,j}$'s. We claim that each $E_k,1\le k\le d_\pi$, consists of the same number of $C_{i,j}$'s.
  
    Assume that $q$ is the greatest number of $\phi$-classes $C_{i,j}$ contained in a single $\pi$-class. We may let $E_1$ be such a $\pi$-class. Again by \hyperref[thm:classes_and_relative_classes]{Theorem \ref{thm:classes_and_relative_classes}} we may assume that $E_1=\bigcup_{j=1}^qC_{1,j}$. For the convenience put $C_j=C_{1,j}$ and $x_j=x_{1,j}$ for $1\le j\le d_\phi$. There are $(\pi,m_j)$-bridges $\vec{x}_j$ from $x_j$ to $x_{j+1}$ for $1\le j\le q-1$ and some $m_j$'s, $n_j$'s with $1<m_1<n_1<m_2<\cdots<m_{q-1}<n_{q-1}$.
  
    Consider blocks $u$ and $u_j=x_j|_{[1,l]},j=1,\cdots,q,$ of length $l>n_{q-1}$ given by
    \[
      u|_t=\begin{cases}
        x_1|_t,&t\in[1,m_1]\\
        \vec x_1|_t,&t\in(m_1,n_1)\\
        x_2|_t,&t\in[n_1,m_2]\\
        \vec x_2|_t,&t\in(m_2,n_2)\\
        \cdots\\
        \vec x_{q-1}|_t,&t\in(m_{q-1},n_{q-1})\\
        x_q|_t,&t\in[n_{q-1},l]
        \end{cases}.
    \]
    The block $u$ starts from $u_1$ and passes across all the $u_j$'s and ends in $u_q$. Later $u$ will take a role of $\pi$-bridges connecting distinct $\phi$-classes. Let $w=\phi(u)$.
  
    Both to the right and to the left of $y|_{[1,l]}=w$ in $y$ which is bi-transitive, there occur magic blocks of $\phi$ infinitely often. Shifting to the right we may assume that for some integers $1\le l'<L',L\ge L'+l$ each of $y|_{[1,l']},y|_{(l',L']}$ and $y|_{(L'+l,L]}$ is a magic block of $\phi$ and $y|_{(L',L'+l]}=w$. The extension $\bar{w}=y|_{[1,L]}$ of $w$ is another magic block of $\phi$ with magic  coordinates $s\le l'$ and $s'>L'+l$ so that $w$ is a subblock of $\hat{w}=\bar{w}|_{[s,s']}$. Note that $\hat{w}$ contains a magic block $\bar{w}|_{(l',L']}$ of $\phi$ and $w$ as its subword. So $\hat{w}$ itself is a magic block of $\phi$ which is an extension of $w$.
  
    Let $\bar{u}_j=x_j|_{[1,L]},\hat{u}_j=\bar{u}_j|_{[s,s']}$ for $1\le j\le q$, and $\hat{l}=|\hat{w}|=s'-s+1$. Fix a bi-transitive point $\bar{x}_1$ from some $\phi$-class $C_{i',\cdot}$ over $y_{i'},1\le{i'}\le d_\psi$, that is not contained in $E_1$ and some $t\in\mathbb Z$ with $\bar{x}_1|_{[t+1,t+L]}=\bar{u}_1$. Since $\phi(\bar{x}_1|_{[t+1,t+L]})=\bar{w}$ has magic coordinates $s$ and $s'$ we have that $\phi^{-1}(y_{i'})|_{t+s}=\phi^{-1}(\bar{w})|_s$ and $\phi^{-1}(y_{i'})|_{t+s'}=\phi^{-1}(\bar{w})_{s'}$. Given $2\le j\le q$ take two $\phi$-preimages $x$ and $x'$ of $y_{i'}$ with $x|_{t+s}=\bar{u}_j|_s=\hat{u}_j|_1$ and $x'|_{t+s'}=\bar{u}_j|_{s'}=\hat{u}_j|_{\hat{l}}$. Then $x|_{(-\infty,t+s)}\hat{u}_jx'|_{(t+s',\infty)}$ is a $\phi$-preimage of $y_{i'}$ in $X$ which is 1-step. As a result $\phi^{-1}(y_{i'})$ contains some $\bar{x}_j$ with $\bar{x}_j|_{[t+s,t+s']}=\hat{u}_j$.
    
    Now find a partition $\mathcal{P}$ of $\phi^{-1}(\hat{w})$, given by \hyperref[thm:block_partition]{Theorem \ref{thm:block_partition}}, such that whenever $y_i|_{[t+1,t+\hat{l}]}=\hat{w}$ for some $1\le i\le d_\psi$ and $t\in\mathbb{Z}$ we have a bijection $\rho_{i,t}:\llbracket y_i\rrbracket_\phi\to\mathcal{P}$ with $C_{i,j}|_{[t+1,t+\hat{l}]}\subset\rho_{i,t}(C_{i,j}),1\le j\le d_\phi$. With $i$ and $t$ fixed all $\hat{u}_j,1\le j\le q$, belong to distinct elements $\rho_{i,t}(C_{i,j})$ of $\mathcal{P}$ as they lie in $C_j|_{[s,s']}$, respectively.  
    
    Since $\mathcal{P}$ is a partition, for any $1\le i\le d_\psi,1\le j\le d_\phi$ and $t\in\mathbb{Z}$ there is at most one $1\le k\le q$ with $\hat{u}_k$ in $C_{i,j}|_{[t+1,t+\hat{l}]}$: If some $x,x'$ in $C_{i,j}$ satisfy that $\phi(x|_{[t+1,t+\hat{l}]})=\phi(x'|_{[t+1,t+\hat{l}]})=\hat{w}$ and $x|_{[t+1,t+\hat{l}]}=\hat{u}_k$, then $x'|_{[t+1,t+\hat{l}]}\ne\hat{u}_{k'}$ for all $k'\ne k$ since $\hat{u}_k$ and $\hat{u}_{k'}$ belong to distinct elements of $\mathcal{P}$ while $x|_{[t+1,t+\hat{l}]}$ and $x'|_{[t+1,t+\hat{l}]}$ to the unique element $\rho_{i,t}(C_{i,j})$ of $\mathcal{P}$ that includes $C_{i,j}|_{[t+s,t+s']}$. It follows that $\bar{x}_1,\cdots,\bar{x}_q$ lie in all distinct $\phi$-classes over $y_{i'}$.
    
    Using $u$ we easily get $(\pi,t)$-bridges from $\bar{x}_1$ to $\bar{x}_2$ to $\cdots$ to $\bar{x}_q$ as $\hat{u}_j$ contains $u_j$ for any $1\le j\le q$. By \hyperref[thm:class_transitive_point]{Theorem \ref{thm:class_transitive_point}} $[\bar{x}_1]_\phi,\cdots,[\bar{x}_q]_\phi$ are contained in a single $\pi$-class. Combining the maximality of $q$ with this fact $[\bar{x}_1]_\pi$ turns out to be exactly $\bigcup_{j=1}^q[\bar{x}_j]_\phi$.
  
    We have seen that $d_\phi$ $\phi$-classes over each $y_i$ are partitioned into $d_\phi/q$ sets each of which consists of $q$ $\phi$-classes constituting a $\pi$-class taken into union. So over $y_i$ $X$ has $d_\phi/q$ $\pi$-classes. According to \hyperref[thm:classes_and_relative_classes]{Theorem \ref{thm:classes_and_relative_classes}} these happen to be also $\phi/\psi$-classes. Hence $d_\phi/q=d_{\phi/\psi}$. At last, counting all $\pi$-classes over all $y_i$ we get $d_\pi=d_\psi d_{\phi/\psi}$.
  \end{proof}

  If $Z$ is of finite type and $\varphi$ is another factor code from $Z$, then
  \[ \begin{split}
    d_{\varphi\circ\psi\circ\phi}&=d_\varphi d_{\psi\circ\phi/\varphi}\\
    &=d_{\varphi\circ\psi}d_{\phi/\varphi\circ\psi}\\
    &=d_\varphi d_{\psi/\varphi}d_{\phi/\varphi\circ\psi}
  \end{split}\]
  and $d_{\psi\circ\phi/\varphi}=d_{\psi/\varphi}d_{\phi/\varphi\circ\psi}$.
  
  We finish the paper with \hyperref[cor:eq_multiplicity]{Corollary \ref{cor:eq_multiplicity}}. It tells in a broader context why the multiplicative equality for the degrees of finite-to-one factor codes holds, and shows that \hyperref[ex:multiplication_fails]{Example \ref{ex:multiplication_fails}} indeed barely slips out of the conditions for the equality.

  \begin{corollary}\label{cor:eq_multiplicity}
    Let $X$ and $Y$ be 1-step irreducible shifts of finite type and $Z$ a sofic shift with 1-block factor codes $\phi:X\to Y,\psi:Y\to Z$ and $\pi=\psi\circ\phi$. For finite-to-one $\psi$ or $\phi$ with $d_\phi=1$ we have that $[x]_\phi=[x]_{\phi/\psi}=[x]_\pi$ and $d_\pi=d_\psi d_\phi$.
  \end{corollary}
  \begin{proof}
    In case $d_\phi=1$, as relative degrees are no greater than degrees, $d_{\phi/\psi}=d_\phi=1$. By \hyperref[thm:division_multiplicity]{(\ref{thm:division_multiplicity})}, $d_\pi=d_\psi d_{\phi/\psi}=d_\psi d_\phi=d_\psi$.
  
    Assume $\psi$ to be finite-to-one. Choose $z$ in $Z$, $C$ in $\llbracket z\rrbracket_\psi$, $y$ in $C$ and $x,x'$ in $\phi^{-1}(y)$. If $x\sim_\pi x'$, then $X$ has a $\pi$-transition $\{\vec x_m\}$ from $x$ to $x'$. It is immediate that all $\phi(\vec x_m)$ and $\phi(x)=\phi(x')=y$ are bi-asymptotic. Since $\psi(\phi(\vec x_m))=\psi(y)=z$ and $\psi$ is finite-to-one, we have that $\phi(\vec x_m)=y$ and that $x\sim_\phi\vec x_m\sim_\phi x'$ for all $m$. Therefore $[x]_\pi=[x]_{\phi/\psi}=[x]_\phi$ and $d_\phi=d_{\phi/\psi}$.
  \end{proof}

\begin{ack*}
  The author was supported by Fondecyt project 3130718. The author would like to thank Mahsa Allahbakhshi and Uijin Jung.
\end{ack*}

\bibliographystyle{amsplain}
\bibliography{ref}

\end{document}